\documentclass[12pt]{article}
\usepackage{amsmath,amsthm,amssymb,color}
\input amssym.def
\input amssym.tex
\date{}

\newtheorem{lemma}{Lemma}[section]

\newtheorem{theorem}[lemma]{Theorem}
\newtheorem{corollary}[lemma]{Corollary}
\newtheorem{remark}[lemma]{Remark}

\begin{document}

\title{{\bf On Hardy $q$-inequalities}}

\author {Lech Maligranda, Ryskul Oinarov and Lars-Erik Persson}

\date{}

\maketitle

\renewcommand{\thefootnote}{\fnsymbol{footnote}}


\begin{abstract}
\noindent {\small Some $q$-analysis variants of Hardy type inequalities of the form
\begin{equation} \label{1}
\int_0^b \left (x^{\alpha-1} \int_0^x t^{- \alpha} f(t) \, d_qt \right)^p d_qx \leq C \, \int_0^b f^p(t) \, d_qt
\end{equation}
with sharp constant $C$ are proved and discussed. A similar result with the Riemann-Liouville operator
involved is also proved. Finally, it is pointed out that by using these techniques we can also obtain some 
new discrete Hardy and Copson type inequalities in the classical case.}
\end{abstract}

%%%%%%%%%%%%%%%%%%%%%%%%%%%%%%%%%%%%%%%%%%%%%%%%%%
\begin{section}
{\bf Introduction and preliminaries}
\end{section}

In recent years quantum calculus ($q$-calculus) has been actively developed. Many continuous scientific
problems have their discrete versions by using the so-called $q$-calculus. This $q$-calculus has
numerous applications in combinatorics, special functions, fractals, dynamical systems, number theory,
computational methods, quantum mechanics, information technology, etc. (see \cite{Ba05}, \cite{Er00},
\cite{Er12}, \cite{Ex83}, \cite{KC02}).

At present $q$-analogues of many inequalities from the classical analysis have been established but not 
$q$-inequalities of Hardy type (see, e.g., \cite{Ga04}, \cite{MQ09} and \cite{Kr11}, \cite{SRM09}, \cite{Su11}).

The Hardy inequality and its various generalizations play an important role in classical analysis.
Therefore during the last fifty years a huge amount of papers has been devoted to Hardy and Hardy type
inequalities in various spaces. The main results and their applications in classical analysis are given in
the books \cite{KP03} and \cite{KMP07}.

The main aim of this paper is to establish $q$-analogue of the classical Hardy type
inequalities
\begin{equation}\label{1.1}
\int_0^\infty \left(x^{\alpha -1}\int_0^x t^{-\alpha}f(t) \, dt \right)^p \, dx < \left(\frac{p}{p-\alpha p -
1}\right)^p \int_0^\infty f^p(t) \, dt, ~ f \geq 0,
\end{equation}
where $\alpha < 1-\frac{1}{p}$ with either $p \geq 1$ (unless $f \equiv 0$) or $p < 0$ and $f > 0$ 
and (with the Riemann-Liouville operator involved)
\begin{equation}\label{1.2}
\int_0^\infty
\left(\frac{1}{x^\alpha\Gamma(\alpha)}\int_0^x (x-t)^{\alpha-1} f(t) \, dt \right)^p \, dx <
\left[\frac{\Gamma\left(1-\frac{1}{p}\right)} {\Gamma\left(\alpha+1-\frac{1}{p}\right)} \right]^p \int_0^\infty f^p(t) \, dt, ~ f \geq 0,
\end{equation}
where $p>1, \alpha>0$, unless $f \equiv 0$ and with the best constant.
For $\alpha=0$ inequality (\ref{1.1}) becomes the classical Hardy inequality
\begin{equation}\label{1.3}
\int_0^\infty \left( \frac{1}{x}\int_0^x f(t) \, dt \right)^p dx < \left(\frac{p}{p-1}\right)^p
\int_0^\infty f^p(t) \, dt, \ f \geq 0, f \not \equiv 0,
\end{equation}
and its corresponding discrete version reads
\begin{equation}\label{1.4}
\sum_{n=1}^\infty\left(\frac{1}{n} \sum_{k=1}^n
a_k\right)^p<\left(\frac{p}{p-1}\right)^p \sum_{n=1}^\infty
a^p_n,\ p>1, \ a_n\geq 0, a_n \not \equiv 0.
\end{equation}
All these estimates have numerous applications in analysis. We will prove the $q$-estimate (\ref{1}) of Hardy type
(\ref{1.1}) for $b = \infty$ and $b = 1$ with the sharp constants, and also the $q$-analogue of the estimate (\ref{1.2})
with the sharp constant.

The paper is organized in the following way: after definitions and notations below, in Section 2 we prove
the $q$-analogue of the inequality (\ref{1.1}), that is, inequality (\ref{1}) for $b = \infty$ and $b = 1$ with the sharp
constants.

In Section 3 we define a fractional $q$-analogue of the Riemann-Liouville operator $I^{\alpha}_q$ and prove a
$q$-analogue of inequality (\ref{1.2}) with the sharp constant.

Finally, in Section 4 we are pointing out that using techniques of $q$-calculus we can also obtain some
new discrete Hardy, Copson and matrix type inequalities in the classical case.

We now present some notations and definitions from the $q$-calculus, which are necessary
for understanding this paper. They are taken mainly from the book  \cite{KC02}.

Let  $0 < q < 1$ be fixed. The {\it definite $q$-integral} or the {\it $q$-Jackson integral}
(see \cite{Ja10} and \cite{KC02}) of a function $f: [0, b) \rightarrow {\mathbb R}, 0 < b \leq \infty,$
is defined by the following formula
\begin{equation}\label{2.1}
\int_0^x f(t) \, d_qt = (1 - q) \,x\, \sum_{k=0}^\infty q^k f(q^k x),\ {\rm for} \,\, x\in (0, b],
\end{equation}
and the {\it improper $q$-integral} of a function $f: [0, \infty) \rightarrow {\mathbb R}$ by the relation
\begin{equation}\label{2.2}
\int_0^\infty f(t) \, d_qt = (1-q)\sum_{k=-\infty}^\infty q^k f(q^k),
\end {equation}
provided that the series on the right hand sides of (\ref{2.1}) and (\ref{2.2}) converge absolutely.

For $0 < a < b \leq \infty$ we define the $q$-integral
$$
\int_a^b f(t) \,d_qt = \int_0^b f(t) \,d_qt - \int_0^a f(t) \,d_qt.
$$
In particular, for $x\in(0,\infty)$, it yields that
\begin{equation}\label{2.3}
\int_x^\infty f(t) \, d_qt = \int_0^\infty f(t) \, d_qt - \int_0^x f(t) \, d_qt.
\end{equation}
In the theory of $q$-analysis the $q$-analogue $[\alpha]_q$ of a number $\alpha\in
\mathbb{R}$ is defined by
\begin{equation}\label{2.4}
 [\alpha]_q = \frac{1-q^\alpha}{1-q}.
 \end{equation}

%%%%%%%%%%%%%%%%%%%%%%%%%%%%%%%%%%%%%% Section 2
\begin{section}
{\bf The Hardy inequality in $q$--analysis}
\end{section}

We consider the $q$-integral analogue of the Hardy inequality of the form (\ref{1.1}). Our first main
result in this section reads:

%%%%%%%%%%%%%%%%%%%%%%% Thm 2.1
\begin{theorem} \label{th3.1}
Let $ \alpha < \frac{p-1}{p}$. If either $1 \leq p < \infty$ and $f \geq 0$ or $p<0$ and $f > 0$, then the inequality
\begin{equation}\label{3.1}
\int_0^\infty x^{p(\alpha -1)} \left(\int_0^x t^{-\alpha} f(t) \,d_qt\right)^p d_qx \leq C\int_0^\infty f^p(t) \,d_qt,
\end{equation}
holds with the constant
\begin{equation}\label{3.2}
C = \frac{1}{[\frac{p-1}{p}-\alpha]_q^p} \,.
\end{equation}
In the case when $0<p<1$ the inequality (\ref{3.1}) for $f \geq 0$ holds in the reverse
direction with the constant (\ref{3.2}). Moreover, in all three cases the
constant (\ref{3.2}) is best possible.
\end{theorem}
%%%%%%%%%%%%

\begin{proof} Let $1<p<\infty$. Consider the estimate (\ref{3.1}) based on the definitions
(\ref{2.1}) and (\ref{2.2}). We have
\vspace{-3mm}
\begin{eqnarray*}
L(f): &=&
\int_0^\infty x^{p(\alpha -1)} \left(\int_0^x t^{-\alpha} f(t) \, d_qt \right)^p d_qx\\
&=&
\int_0^\infty x^{p(\alpha -1)} \left((1-q)
\sum\limits_{i=0}^\infty x^{1-\alpha}q^{(1-\alpha)i}f(xq^i)\right)^p d_qx \\
&=&
(1-q)^{p+1}\sum_{j=-\infty}^{\infty}q^{jp(\alpha-1)}\left(\sum\limits_{i=0}^\infty
q^{(i +j)(1-\alpha)}f(q^{i+j})\right)^p q^j \\
&=&
(1-q)^{p+1}\sum_{j=-\infty}^{\infty}q^{j[p(\alpha-1)+1]}
\left(\sum\limits_{i=j}^\infty
q^{i(1-\alpha)}f(q^{i})\right)^p\equiv (1-q)^{p+1}I^p.
\end{eqnarray*}
Let $g=\{g_k\}_{k=-\infty}^\infty\in l_{p'}(\mathbb{Z}),$ $g\ge
0, \|g\|_{l_{p'}}=1$, where $\frac{1}{p}+\frac{1}{p'}=1,$ Moreover, let
$\theta(z)$ be Heaviside's unit step function, that is, $\theta(z)=1$
for $z \geq 0$ and $\theta(z)=0$ for  $z < 0$. Then, based on the duality
principle in $l_p(\mathbb{Z}),$ $p>1$, and the H\"{o}lder-Rogers inequality 
(cf. \cite{Ma98} for the explanation why not only H\"older name should be here), we find that
\begin{eqnarray*}
I
&=&
\sup\limits_{\|g\|_{l_{p'}}=1} \sum_i\sum_j g_jq^{j\left(\alpha-\frac{1}{p'}\right)}
\theta(i-j)q^{i\left(\frac{1}{p'}-\alpha\right)}q^\frac{i}{p}f(q^i) \\
&\leq&
\sup\limits_{\|g\|_{l_{p'}}=1}\left(\sum_i\sum_j g_j^{p'}
q^{j\left(\alpha-\frac{1}{p'}\right)}\theta(i-j)q^{i\left(\frac{1}{p'}-\alpha\right)}\right)^{\frac{1}{p'}} \\
&\times&
\left(\sum_i\sum_j f^p(q^i)q^i q^{\left(\frac{1}{p'}-\alpha\right)i}\theta(i-j)q^{j\left(\alpha-\frac{1}{p'}\right)}\right)^\frac{1}{p} \\
&\leq&
\sup\limits_{\|g\|_{l_{p'}}=1}\left(\sum_j g_j^{p'}
q^{j\left(\alpha-\frac{1}{p'}\right)}\sum_{i=j}^\infty
q^{i\left(\frac{1}{p'}-\alpha\right)}\right)^{\frac{1}{p'}} \\
&\times&
\left(\sum_i f^p(q^i)q^i q^{i\left(\frac{1}{p'}-\alpha\right)}
\sum_{j=-\infty}^i
q^{j\left(\alpha-\frac{1}{p'}\right)}\right)^\frac{1}{p} =
\sup\limits_{\|g\|_{l_{p'}}=1} I_1(g) \,I_2(f).
\end{eqnarray*}
Since
\begin{eqnarray*}
I_1^{p'}(g)
&=&
\sum_j g_j^{p'} q^{j\left(\alpha-\frac{1}{p'}\right)}q^{j\left(\frac{1}{p'}-\alpha\right)}
\sum_{i=0}^\infty q^{i\left(\frac{1}{p'}-\alpha\right)} =
\frac{1}{1-q^{\frac{1}{p'}-\alpha}}\sum_j g_j^{p'} \\
&=&
\frac{1}{1-q^{\frac{1}{p'}-\alpha}} \, \| g \|_{l_{p'}}^{p'} =
\frac{1}{(1-q)\left[\frac{p-1}{p}-\alpha\right]_q}  \, \| g \|_{l_{p'}}^{p'},
\end{eqnarray*}
and
\begin{eqnarray*}
I_2^p(f)
&=&
\sum_i f^p(q^i)q^i q^{i\left(\frac{1}{p'}-\alpha\right)}
q^{i\left(\alpha-\frac{1}{p'}\right)} \sum_{j=0}^\infty
q^{j\left(\frac{1}{p'}-\alpha\right)} =
\frac{1}{1-q^{\frac{1}{p'}-\alpha}}\sum_i q^if^p(q^i) \\
&=&
\frac{1}{(1-q)^2\left[\frac{p-1}{p}-\alpha\right]_q}
\int\limits_0^\infty f^p(t) \,d_qt
\end{eqnarray*}
it follows that
$$
I^p\leq \frac{1}{(1-q)^{p+1}\left[\frac{p-1}{p}-\alpha\right]_q^p}
\int\limits_0^\infty f^p(t) \,d_qt.
$$
Putting the above calculations together we deduce that
\begin{equation*}
\int_0^\infty x^{p(\alpha -1)}\left(\int_0^x t^{-\alpha}f(t) \,d_qt\right)^p d_qx =
(1-q)^{p+1}I^p \leq
\frac{1}{\left[\frac{p-1}{p}-\alpha\right]_q^p} \int_0^\infty f^p(t) \, d_qt,
\end{equation*}
which means that inequality (\ref{3.1}) holds with constant (\ref{3.2}).

Now, we will show that the constant (\ref{3.2}) is best possible. For $\beta > -\frac{1}{p}$ let
$f_\beta(t) = t^\beta \chi_{(0, 1]}(t), t > 0$. Then
\begin{equation*}
\int_0^\infty f_\beta^p(t) \, d_qt = (1-q)\sum_{i=-\infty}^\infty q^i f^p_\beta (q^i)= (1-q)\sum_{i=0}^\infty q^i  q^{p\beta
i} = \frac{1-q}{1-q^{1+p\beta}} = \frac{1}{[1+p\beta]_q},
\end{equation*}
and
\begin{eqnarray*}
L(f_{\beta})
&=&
\int_0^\infty  x^{p(\alpha -1)}\left(\int_0^x t^{-\alpha} f_\beta(t) \, d_qt\right)^p d_qx \\
&=&
(1-q)\sum_{j=-\infty}^{\infty}q^{j[p(\alpha-1)+1]}\left((1-q)\sum_{i=0}^\infty q^{(i+j)(1-\alpha)}f_\beta(q^{i+j})\right)^p \\
&\geq&
(1-q)^{p+1}\sum_{j=0}^\infty q^{j[p(\alpha-1)+1]}\left(\sum_{i=j}^\infty q^{i(1-\alpha)}q^{i\beta}\right)^p \\
&= &
(1-q)^{p+1}\sum_{j=0}^\infty q^{j(1+p\beta)}\left(\sum_{i=0}^\infty q^{i(1-\alpha+\beta)}\right)^p =
\left(\frac{1-q}{1-q^{1-\alpha+\beta}}\right)^p \cdot \frac{1}{[1+p\beta]_q}.
\end{eqnarray*}
Since
\begin{equation*}
\sup\limits_{\beta > - \frac{1}{p}} \frac{1}{1-q^{1-\alpha+\beta}} =\frac{1}{1-q^{1-\frac{1}{p}-\alpha}}
= \frac{1}{1-q^{\frac{p-1}{p}-\alpha}},
\end{equation*}
it follows that for the best constant $C$ in (\ref{3.1}) the following estimate
is valid
\begin{equation*}
C \geq \sup\limits_{\beta > - \frac{1}{p}} \frac{L(f_{\beta})}{\int_0^\infty f_\beta^p(t) \, d_qt }
\geq \sup\limits_{\beta > - \frac{1}{p}} \left(\frac{1-q}{1-q^{1-\alpha+\beta}}\right)^p =
\frac{1}{\left[\frac{p-1}{p}-\alpha\right]_q^p},
\end{equation*}
which shows that the constant (\ref{3.2}) is sharp constant in (\ref{3.1}).

If $p=1$, then
\begin{eqnarray*}
L(f) &=&
(1-q)^{2}\sum_{j=-\infty}^{\infty}q^{j\alpha}
\sum\limits_{i=j}^\infty q^{i(1-\alpha)}f(q^{i}) =
(1-q)^2\sum_{i=-\infty}^{\infty}q^{i(1-\alpha)}f(q^{i})\sum\limits^{i}_{j=-\infty}q^{j\alpha} \\
&=&
(1-q)^2\sum_{i=-\infty}^{\infty}q^{i}f(q^{i})\sum\limits_{j=0}^{\infty}q^{-j\alpha} = 
\frac{1}{[-\alpha]_q}\int\limits_0^\infty f(t)d_qt.
\end{eqnarray*}

%%%%%%%%%%%%%%%%%
Let $p<0$ and $f > 0$. If we denote $\mu=\frac{1}{p'}\left(\frac{1}{p'}-\alpha\right)$, then
\begin{eqnarray*}
L(f)
&=&
\int\limits_0^\infty x^{p(\alpha -1)}\left(\int\limits_0^x t^{-\alpha} f(t) \,d_qt\right)^p d_qx \\
&=&
(1-q)^{p+1}\sum_{j=-\infty}^{\infty}q^{j[p(\alpha-1)+1]}
\left(\sum\limits_{i=j}^\infty q^{i(1-\alpha)}f(q^{i})\right)^p \\
&=&
(1-q)^{p+1}\sum_{j=-\infty}^{\infty}q^{j[p(\alpha-1)+1]}
\left(\sum\limits_{i=j}^\infty q^{i
\mu}q^{i(1-\alpha-\mu)}f(q^{i})\right)^p.
\end{eqnarray*}
Taking into account the assumption $p<0$ and the fact that then the H\"{o}lder-Rogers inequality holds in the
reverse direction in this case we obtain
\begin{eqnarray*}
L(f)
&\leq&
(1-q)^{p+1}\sum_{j=-\infty}^{\infty}q^{j[p(\alpha-1)+1]}
\left( \sum\limits_{i=j}^\infty q^{ip'\mu}\right)^{p-1}
 \sum\limits_{i=j}^\infty q^{ip(1-\alpha-\mu)}f^p(q^{i} )\\
&=&
(1-q)^{p+1}  \left( \sum\limits_{i=0}^\infty
q^{ip'\mu}\right)^{p-1}
\sum_{j=-\infty}^{\infty}q^{j[p(\alpha-1)+1+p\mu]}
\sum\limits_{i=j}^\infty q^{ip(1-\alpha-\mu)} f^p(q^{i})\\
&=&
\frac{(1-q)^{p+1}}{\left(1-q^{\frac{p-1}{p}-\alpha}\right)^{p-1}
} \sum_{i=-\infty}^{\infty} f^p(q^{i})
q^{i\left(1+\frac{1}{p'}-\alpha\right)} \sum_{j=-\infty}^i
q^{j\left(\alpha-\frac{1}{p'}\right)} \\
&=&
\frac{(1-q)^{p+1}}{\left(1-q^{\frac{p-1}{p}-\alpha}\right)^{p-1}
} \sum_{j=0}^\infty q^{j\left(\frac{1}{p'}-\alpha\right)}
\sum_{i=-\infty}^\infty q^i f^p(q^i) \\
&=&
\left(\frac{1-q}{1-q^{\frac{p-1}{p}-\alpha}}\right)^p
\int\limits_0^\infty f^p (t) \,d_qt =
\left[\frac{p-1}{p} - \alpha\right]_q^{-p} \int_0^\infty f^p(t) \,d_qt.
\end{eqnarray*}
This implies that the inequality (\ref{3.1}) holds with the constant $C$ in (\ref{3.2}). Now, we will give 
lower estimate for the best constant $C$ in inequality (\ref{3.1}). 
For $ \alpha-1<\beta_1 < - \frac{1}{p} < \beta_2$ let
$f_{\beta_1, \beta_2}(t ) = t^{\beta_1} \chi_{(0, 1]}(t) + t^{\beta_2} \chi_{(1, \infty)}(t), t > 0$.
Then
\begin{eqnarray*}
\int_0^\infty f_{\beta_1, \beta_2}^p(t) \, d_q t
&=&
(1-q)\left(\sum_{i=-\infty}^{-1} q^{i(1+p\beta_2)}+
\sum_{i=0}^{\infty} q^{i(1+p\beta_1)}\right) \\
&=&
(1-q)\left(\frac{q^{|1+p\beta_2|}}{1-q^{|1+p\beta_2|}}+\frac{1}{1-q^{1+p\beta_1}}\right) : =
F^-(\beta_1,\beta_2)
\end{eqnarray*}
and
\begin{eqnarray*}
L(f_{\beta_1, \beta_2})
&=&
\int_0^\infty x^{p(\alpha -1)} \left(\int\limits_0^x t^{-\alpha} f_{\beta_1, \beta_2}(t) \,d_qt\right)^p d_qx \\
&=&
(1-q)^{p+1}\sum_{i=-\infty}^{\infty}q^{i(1+p(\alpha-1))}\left(\sum_{j=i}^{\infty}q^{j(1-\alpha)} f_{\beta_1, \beta_2}(q^j)\right)^p \\
&=&
(1-q)^{p+1} \left[ \sum_{i=-\infty}^{-1}q^{i(1+p(\alpha-1))}\left(\sum_{j=i}^{-1}q^{j(1-\alpha+\beta_2)}
+\sum_{j=0}^{\infty}q^{j(1-\alpha+\beta_1)}\right)^p\right.\\
&+&
\left.\sum_{i=0}^{\infty}q^{i(1+p(\alpha-1))}\left(\sum_{j=i}^{\infty}q^{j(1-\alpha+\beta_1)}\right)^p \right] \\
&>&
(1-q)^{p+1}\sum_{i=0}^{\infty}q^{i(1+p(\alpha - 1))}\left(\sum_{j=i}^{\infty}q^{j(1-\alpha+\beta_1)}\right)^p \\
&=&
(1-q)^{p+1}\sum_{i=0}^{\infty}q^{i(1+p \beta_1)}\left(\sum_{j=0}^{\infty}q^{j(1-\alpha+\beta_1)}\right)^p \\
&=&
\frac{1-q}{1-q^{1+p \beta_1}}\left(\frac{1-q}{1-q^{1-\alpha+\beta_1}} \right)^p: = F^+(\beta_1, \beta_2).
\end{eqnarray*}
If $C$ is the best constant in (\ref{3.2}), then
\begin{eqnarray*}
C
&\geq&
\sup_{ \alpha-1<\beta_1 < - \frac{1}{p}} \, \lim_{\beta_2 \rightarrow \infty} \frac{F^+(\beta_1, \beta_2)}{F^-(\beta_1,\beta_2)} 
 = \sup_{ \alpha-1<\beta_1 < - \frac{1}{p}} \left(\frac{1-q}{1-q^{1-\alpha+\beta_1}} \right)^p \\
&=&
\left(\frac{1-q}{1-q^{1-\alpha-1/p}}\right)^p = \frac{1}{\left[\frac{p-1}{p}-\alpha\right]_q^p} \,.
\end{eqnarray*}
The last estimate together with the earlier shows that constant (\ref{3.2}) is sharp in all cases.

Finally, we consider the case when  $0<p<1$. Let us denote $\gamma = \frac{p-1}{p}-\alpha$.
For any function $f \geq 0$ for which the right hand side of (\ref{3.1}) is finite, we find that
\begin{eqnarray*}
[\gamma]_q^{-1} \, \int_0^\infty f^p(t) \,d_qt
&=&
\frac{(1-q)^2}{1-q^\gamma}\sum_{j=-\infty}^\infty q^jf^p(q^j) \\
&=&
(1-q)^2\sum_{j=-\infty}^\infty q^jf^p(q^j)\sum_{i=0}^\infty q^{i\gamma}\\
&=&
(1-q)^2\sum_{j=-\infty}^\infty q^jf^p(q^j)\sum_{i=-\infty}^0 q^{-i\gamma}\\
&=&
(1-q)^2\sum_{j=-\infty}^\infty q^{j(1+\gamma)}f^p(q^j)\sum_{i=-\infty}^jq^{-i\gamma}\\
&=&
(1-q)^2 \sum_{i=-\infty}^\infty q^{-i\gamma}\sum_{j=i}^\infty q^{j(1-p)\gamma}q^{jp(1-\alpha)}f^p(q^j) = J.
\end{eqnarray*}
Using the H\"{o}lder-Rogers inequality with powers $\frac{1}{p}$ and $\frac{1}{1-p}$ we obtain
\begin{eqnarray*}
J
&\leq&
(1-q)^2 \sum_{i=-\infty}^\infty q^{-i\gamma}
\left(\sum_{k=i}^\infty q^{k\gamma}\right)^{1-p}\left(\sum_{j=i}^\infty q^{j(1-\alpha)}f(q^j)\right)^p\\
&=&
[\gamma]_q^{p-1} \, (1-q)^{p+1} \sum_{i=-\infty}^\infty q^{-ip\gamma} \left(\sum_{j=i}^\infty q^{j(1-\alpha)}f(q^j)\right)^p\\
&=&
[\gamma]_q^{p-1} \, (1-q)\sum_{i=-\infty}^\infty q^iq^{ip(\alpha-1)}\left((1-q)q^i\sum_{j=0}^\infty q^j q^{-(i+j)\alpha}
f(q^{i+j})\right)^p\\
&=&
[\gamma]_q^{p-1} \, \int_0^\infty x^{p(\alpha -1)} \left(\int_0^x t^{-\alpha} f(t) \, d_qt\right)^p d_qx,
\end{eqnarray*}
which means that the following inequality holds
\begin{equation}\label{3.11}
\int_0^\infty f^p(t) \, d_qt \leq \big[ \gamma \big]_q^p
\int_0^\infty x^{p(\alpha -1)}\big(\int_0^x t^{-\alpha} f(t) \, d_qt \big)^p d_qx
\end{equation}
for all functions $f \geq 0$ for which the left hand side of (\ref{3.11}) is finite.

Next, we show that the constant $\big[ \gamma \big]_q^p = \big[ \frac{p-1}{p} - \alpha \big]_q^p $ in (\ref{3.11}) is
sharp. For $\alpha-1 < \beta < - \frac{1}{p}$ let
$f_\beta(t) = t^\beta \chi_{[1, \infty)}(t), t > 0$.
Then
\begin{eqnarray*}
\int_0^\infty  f_\beta^p(t) \, d_qt
&=&
(1-q)\sum_{i=-\infty}^\infty q^i f_\beta^p(q^i)
= (1-q)\left[ \sum_{i=-\infty}^0 q^i f_\beta^p(q^i)+ \sum_{i=1}^\infty q^i f_\beta^p(q^i) \right] \\
&=&
(1-q)  \sum_{i=-\infty}^0 q^{i(1+p\beta)} =(1-q)  \sum_{i=0}^\infty q^{i|1+p\beta|} = \frac{1-q}{1-q^{|1+p\beta|}}
\end{eqnarray*}
and
\begin{eqnarray*}
L(f_{\beta})
&=&
\int_0^\infty x^{p(\alpha -1)} \left(\int_0^x t^{-\alpha} f_\beta(t) \, d_qt \right)^p d_qx \\
&=&
(1-q) \sum_{j=-\infty}^\infty q^{j[p(\alpha-1)+1]} \left((1-q)q^j\sum_{i=0}^\infty  q^iq^{-(i+j)\alpha}f_\beta(q^{i+j})\right)^p \\
&=&
(1-q)^{p+1}\left[\sum_{j=-\infty}^0 q^{j[p(\alpha-1)+1]} \left(\sum_{i=j}^\infty  q^{i(1-\alpha)}f_\beta(q^{i})\right)^p\right. \\
&+&
\left.\sum_{j=1}^\infty q^{j[p(\alpha-1)+1]} \left(\sum_{i=j}^\infty  q^{i(1-\alpha)}f_\beta(q^{i})\right)^p \right] \\
&=&
(1-q)^{p+1}\sum_{j=-\infty}^0 q^{j[p(\alpha-1)+1]} \left(\sum_{i=j}^0 q^{i(1-\alpha+\beta)}\right)^p \\
&=&
(1-q)^{p+1}\sum_{j=-\infty}^0q^{j[p(\alpha-1)+1]} q^{jp(1-\alpha+\beta)}\left(\sum_{i=0}^{-j} q^{i(1-\alpha+\beta)}\right)^p \\
&\leq&
\frac{(1-q)^{p+1}}{(1-q^{1-\alpha+\beta})^p}\sum_{j=-\infty}^0q^{j[1+p\beta]}=\frac{1-q}{1-q^{|1+p\beta|}}
\left( \frac{1-q}{1-q^{1-\alpha+\beta}}\right)^p.
\end{eqnarray*}
If inequality (\ref{3.11}) holds with the best constant $C > 0$, then
\begin{eqnarray*}
C
&\geq&
\sup\limits_{\beta \in (\alpha-1, -\frac{1}{p})} \frac{\int_0^\infty  f_\beta^p(t) \, d_qt }{L(f_{\beta}) }
\geq \sup\limits_{\beta \in (\alpha-1, -\frac{1}{p})} \left( \frac{1-q^{1-\alpha+\beta}}{1-q}\right)^p \\
&=&
\left(
\frac{1-q^{\frac{p-1}{p}-\alpha}}{1-q}\right)^p = \left[\frac{p-1}{p}-\alpha\right]_q^p = [\gamma]_q^p,
\end{eqnarray*}
and this shows that the constant $[\gamma]_q^p$ in (\ref{3.11}) is sharp. The proof of Theorem \ref{th3.1} is complete.
\end{proof}

%%%%%%%%%%% Remark 2.2
\begin{remark} \label{r2.2} Constant in $q$-analogue of inequality (\ref{1.1}) is smaller than the one in  (\ref{1.1}). 
In fact, if $\alpha < 1 - 1/p$ with $p \geq 1$ or $p < 0$, then 
\begin{equation} \label{R2.2}
\frac{1}{\left[\frac{p-1}{p}-\alpha\right]_q} < \frac{p}{p - \alpha p - 1} \,\,\, {\it for} \,\, \alpha > - 1/p.
\end{equation}
Inequality (\ref{R2.2}) is reversed for $\alpha < - 1/p$. For $\alpha = - 1/p$ both sides in (\ref{R2.2}) are equal to $1$.
\end{remark}

Estimate (\ref{R2.2}) means that $\frac{1-q}{1-q^{\frac{p-1}{p} - \alpha}} < \frac{p}{p - \alpha p - 1}$ for any $0 < q < 1$,
which is true since the function $h(q): = \frac{p(1-q^{\frac{p-1}{p} - \alpha})}{p - \alpha p - 1} + q - 1$ has derivative
$h^{\prime}(q) = - q^{-\frac{1}{p} - \alpha} + 1 < 0$  for $\alpha > - 1/p$, and so $h(q) > h(1) = 0$.

%%%%%%%%
Next, we consider the Hardy inequality on a finite interval. Without loss of generality we consider
only the interval $[0, 1]$, since in the $q$-integral we are allowed to change variables in the form $z=x l$,
$0 < l < \infty$ (see \cite{KC02}). Therefore, a $q$-integral on the interval $[0, l]$ naturally can be
reduced to a $q$-integral on the interval $[0, 1]$.

Hence, we consider the inequality (\ref{1}) with $b=1$ and formulate our next main theorem in this Section.

%%%%%%%%%%%%%%%%%%% Thm 2.3
\begin{theorem} \label{th3.2}
Let $\alpha<1-\frac{1}{p}$.  If either $1 \leq p < \infty$ and $f \geq 0$ or $p<0$ and $f > 0$, then the 
strict inequality
\begin{equation}\label{18}
\int_0^1x^{p(\alpha -1)} \left(\int_0^x t^{-\alpha} f(t) \,d_qt\right)^p d_qx <  \frac{1}{[\frac{p-1}{p}-\alpha]_q^p}
\int_0^1 f^p(t) \,d_qt,
\end{equation}
holds (unless $f \equiv 0$) and the constant $[\frac{p-1}{p}-\alpha]_q^{-p}$ is sharp.
\end{theorem}

\begin{proof} Theorem \ref{th3.2} can be proved in a similar way as Theorem \ref{th3.1}. Hence, we will
only point out some differences of the corresponding relations. In the case when $p>1$ we have
\begin{eqnarray*}
\int_0^1 x^{p(\alpha-1)} \left(\int_0^x t^{-\alpha} f(t) d_qt \right)^pd_qx
&=&
(1-q)^{p+1}\sum\limits_{j=0}^{\infty}q^{j[p(\alpha-1)+1]}
\left(\sum\limits_{i=j}^{\infty}q^{i(1-\alpha)}f(q^i)\right)^p \\
&=&
(1-q)^{p+1} \, I^p,
\end{eqnarray*}
and

\begin{eqnarray*}
I
&<&
\sup\limits_{||g||_{p'}=1,g\geq 0} \left(\sum\limits_{j=0}^{\infty}g_i^{p'}q^{i(\alpha-\frac{1}{p'})}
\sum\limits_{i=j}^{\infty}q^{i(\frac{1}{p'}-\alpha)}\right)^{\frac{1}{p'}}
\left(\sum\limits_{i=0}^{\infty}f^p(q^i)q^i
q^{i(\frac{1}{p'}-\alpha)}
\sum\limits_{j=-\infty}^iq^{i(\alpha-\frac{1}{p'})}\right)^{\frac{1}{p}} \\
&=&
\sup\limits_{||g||_{p'} = 1,g\geq 0} I_1(g) I_2(f),
\end{eqnarray*}
respectively. If $p=1$, then
\begin{eqnarray*}
\int_0^1 x^{\alpha-1} \int_0^x t^{-\alpha} f(t) d_qt d_qx
&=&
(1-q)^2\sum\limits_{j=0}^{\infty}q^{j\alpha}
\sum\limits_{i=j}^{\infty}q^{i(1-\alpha)}f(q^i) \\
&=&
(1-q)^2 \sum\limits_{i=0}^{\infty}q^{i(1-\alpha)}f(q^i)
\sum\limits_{j=0}^iq^{j\alpha} \\
&=&
(1-q)^2 \sum\limits_{i=0}^{\infty}q^if(q^i)
\sum\limits_{j=0}^iq^{-j\alpha} <
\frac{1}{[-\alpha]_q} \int\limits_{0}^{1} f(t) \, d_qt.
\end{eqnarray*}

The last strict inequalities give the validity of strict inequality (\ref{18}). The best constant in (\ref{18}) can be 
found by using the test functions $f_{\beta}(t) = t^{\beta}$ if $0<t<1$, where $\beta > - \frac{1}{p}$.
In the case when $p<0$ the proof of estimate (\ref{18}) can be done by use of the same method as in Theorem \ref{th3.1}
for $F$. In fact, we have
$$
L(f) < \frac{(1-q)^{p+1}}{\left(1-q^{\frac{p-1}{p}-\alpha}\right)^{p-1} } \sum_{i=0}^{\infty} f^p(q^{i})
q^{i\left(1+\frac{1}{p'}-\alpha\right)} \sum_{j=-\infty}^i
q^{j\left(\alpha-\frac{1}{p'}\right)} = \frac{1}{[\frac{p-1}{p}-\alpha]_q^p} \int\limits_{0}^{1} f^p(t) \, d_qt.
$$
This implies strict inequality in (\ref{18}). In order to obtain lower estimate we consider the test functions
$ f_{\beta}(t) = t^{\beta} \chi_{(0, 1]}(t)$ for $ \alpha-1 < \beta < - \frac{1}{p}$. Then
$$
\int_0^1 f_{\beta}^p(t) \,d_q t =
\frac{1-q}{1-q^{1+p\beta}}: = F^-(\beta)
$$
and
\begin{eqnarray*}
L(f_{\beta})
&=&
\int_0^1 x^{p(\alpha -1)}\left(\int\limits_0^x t^{-\alpha} f_{\beta}(t) \, d_qt\right)^p d_qx \\
&=&
(1-q)^{p+1}\sum_{i=0}^{\infty}q^{i(1+p(\alpha-1))}\left(\sum_{j=i}^{\infty}q^{j(1-\alpha+\beta)}\right)^p  \\
&> &
(1-q)^{p+1}\frac{1}{1-q^{1+p\beta}}\left(\frac{1}{1-q^{1-\alpha+\beta}}\right)^p \\
&=&
\frac{1-q}{1-q^{1+p\beta}}\left(\frac{1-q}{1-q^{1-\alpha+\beta}}\right)^p: =
F^+(\beta).
\end{eqnarray*}
Hence, if $C > 0$ is the best constant in inequality (\ref{18}), then we obtain estimate
\begin{equation*}
C \geq
 \sup_{ \alpha-1<\beta_2 < - \frac{1}{p}} \, \lim_{\beta_1 \rightarrow \infty} \frac{F^+(\beta)}{F^-(\beta)}
= \frac{1}{\left[\frac{p-1}{p}-\alpha\right]_q^p} \,, 
\end{equation*}
and the proof of Theorem \ref{th3.2} is complete.
\end{proof}

Next, we present some corresponding sharp reverse inequalities with additional terms for the case $0 < p < 1$.

%%%%%%%%%%%%%%%%%%%%%%% Thm 2.4
\begin{theorem} \label{th3.3}
Let $0 < p < 1$ and $\alpha < \frac{p-1}{p}$. Then the following strict inequalities hold:
\begin{equation}\label{5.6}
\int_0^1 f^p(t)(1-t^{\frac{p-1}{p}-\alpha}) \, d_qt < C\int_0^1 x^{p(\alpha -1)}\left(\int_0^x t^{-\alpha} f(t) \, d_qt \right)^p d_qx,
\end{equation}
\begin{equation}\label{5.7}
\int_0^1 f^p(t) \, d_qt < C\int_0^1\left(1+\frac{\chi_{(q,1]}(x)}{[\frac{p-1}{p}-\alpha]_q}\right)
x^{p(\alpha -1)}\left(\int_0^x t^{-\alpha} f(t) \, d_qt \right)^p d_qx
\end{equation}
for all functions $f\geq0$ with the finite left hand side of (\ref{5.7}) unless $f \equiv 0$ and with the best constant
\begin{equation}\label{5.8}
C = \left[\frac{p-1}{p}-\alpha\right]_q^p.
\end{equation}
\end{theorem}

%%%%%%%%%%%%%%%%%
\begin{proof}
Let $f\geq0$ and $\int_0^1 f^p(t) \, d_qt < \infty$. Denoting
$\gamma = \frac{p-1}{p}-\alpha$ and $c_q = (1-q)(1-q^{\gamma})$ we obtain
\begin{eqnarray*}
\int_0^1 f^p(t) \, d_qt
&=&
(1-q)\sum_{j=0}^\infty q^jf^p(q^j) = c_q\sum_{j=0}^\infty q^jf^p(q^j)\sum_{i=-\infty}^0 q^{-i\gamma}\\
&=&
c_q\sum_{j=0}^\infty q^{j(1+\gamma)}f^p(q^j)\sum_{i=-\infty}^jq^{-i\gamma} \\
&=&
c_q\sum_{j=0}^{\infty} q^{j(1+\gamma)}f^p(q^j)\sum_{i=0}^j q^{-i\gamma} +
c_q\sum_{i=-\infty}^{-1} q^{-i\gamma}\sum_{j=0}^{\infty} q^{j(1+\gamma)}f^p(q^j) \\
&=&
c_q\sum_{j=0}^{\infty} q^{j(1+\gamma)}f^p(q^j)\sum_{i=0}^j q^{-i\gamma} +
(1-q) q^{\gamma} \sum_{j=0}^{\infty} q^{j(1+\gamma)}f^p(q^j) \\
&<&
c_q\sum_{i=0}^{\infty} q^{-i\gamma}\sum_{j=i}^{\infty} q^{j(1+\gamma)}f^p(q^j)+(1-q)\sum_{j=0}^{\infty}
q^{j(1+\gamma)}f^p(q^j) : = I_1 + I_2.
\end{eqnarray*}
By using the H\"{o}lder-Rogers inequality with powers $\frac{1}{1-p}$ and $\frac{1}{p}$ we can estimate $I_1$ as follows
\begin{eqnarray*}
I_1
&=&
c_q \sum_{i=0}^\infty q^{-i\gamma}\sum_{j=i}^\infty q^{j(1-p)\gamma}q^{jp(1-\alpha)}f^p(q^j) \\
&<&
c_q\sum_{i=0}^\infty q^{-i\gamma} \left(\sum_{k=i}^\infty q^{k\gamma}\right)^{1-p}\left(\sum_{j=i}^\infty
q^{j(1-\alpha)}f(q^j)\right)^p \\
&=&
c_q\left(\frac{1}{1-q^{\gamma}}\right)^{1-p}\sum_{i=0}^\infty
q^{i(1+p(\alpha-1))}\left(\sum_{j=i}^\infty  q^{j(1-\alpha)} f(q^{j})\right)^p \\
&=&
\left(\left[\frac{p-1}{p}-\alpha\right]_q\right)^p\int_0^1x^{p(\alpha -1)}\left(\int_0^x t^{-\alpha} f(t) \, d_qt \right)^p d_qx.
\end{eqnarray*}
Since $I_2 = (1-q)\sum_{j=0}^{\infty} q^{j(1+\gamma)}f^p(q^j) = \int_{0}^{1} t^{\frac{p-1}{p}-\alpha} f^p(t) \, d_qt$ it follows
from the above calculations that estimate (\ref{5.6}) holds with the constant $ C \leq \left[\frac{p-1}{p}-\alpha\right]_q^p$.

Now, we will show also validity of the inequality (\ref{5.7}). For this purpose we estimate $I_2$, using the H\"{o}lder-Rogers
inequality with powers $\frac{1}{1-p}$ and $\frac{1}{p}$, and obtain
\begin{eqnarray*}
I_2
&=&
(1-q)\sum_{j=0}^\infty q^{j(1-p)\gamma}q^{jp(1-\alpha)}f^p(q^j) <
(1-q)\left(\sum_{k=0}^\infty q^{k\gamma}\right)^{1-p}\left(\sum_{j=0}^\infty q^{j(1-\alpha)}f(q^j)\right)^p \\
&=&
(1-q) (1-q^{\gamma})^{p-1} \left(\sum_{j=0}^{\infty} q^{j(1-\alpha)}f(q^j)\right)^p
=
\left[\frac{p-1}{p}-\alpha\right]_q^{p-1}\left(\int_0^1t^{-\alpha} f(t) \, d_qt \right)^p.
\end{eqnarray*}
Hence, again from the above calculations we obtain
\begin{eqnarray*}
\int_0^1f^p(t) \, d_qt
&<&
\big[\frac{p-1}{p}-\alpha\big]_q^p \, \big{[} \int_0^1 x^{p(\alpha -1)}\big(\int_0^x t^{-\alpha} f(t) \, d_qt \big)^p d_qx \\
&+&
\frac{1}{[\frac{p-1}{p}-\alpha]_q}\big(\int_0^1x^{-\alpha}f(x)d_qx\big)^p\big{]}.
\end{eqnarray*}
This means that (\ref{5.7}) holds with the constant (\ref{5.8}). Next, we show that the constant $[\gamma ]_q^p = [\frac{p-1}{p}-\alpha]_q^p$ 
in both of the inequalities (\ref{5.6}) and (\ref{5.7}) is best possible. To see this we consider the function
$f_{\beta}(t) = t^{\beta}$ for $0 < t \leq1$, where $\beta>-\frac{1}{p}$. Then
\begin{equation*}
\int_0^1 t^{\frac{p-1}{p}-\alpha} f_{\beta}^p(t) \, d_qt = \frac{1-q}{1-q^{1+p\beta+\gamma}},
~~ \int_0^1 f_{\beta}^p(t) \, d_qt = \frac{1-q}{1-q^{1+p\beta}},
\end{equation*}
\begin{equation*}
\int_0^1 x^{p(\alpha -1)}\big(\int\limits_0^x t^{-\alpha} f_{\beta}(t) \, d_qt \big)^p d_qx
= \big(\frac{1-q}{1-q^{1-\alpha+\beta}}\big)^p\frac{1-q}{1-q^{1+p\beta}},
\end{equation*}
and
\begin{equation*}
\frac{1}{[\frac{p-1}{p}-\alpha]_q} \big(\int_0^1 t^{-\alpha} f_{\beta}(t) \, d_qt \big)^p
= \frac{1}{[\frac{p-1}{p}-\alpha]_q}\big(\frac{1-q}{1-q^{1-\alpha+\beta}}\big)^p.
\end{equation*}
If $C > 0$ is a sharp constant in inequality (\ref{5.6}), then
$$
C \geq \big(\frac{1-q^{1-\alpha+\beta}}{1-q}\big)^p \, \frac{1-q-(1-q^{1+p\beta})\frac{1-q}{1-q^{1+p\beta+\gamma}}}{1-q},
$$
and by letting $\beta\rightarrow -\frac{1}{p}$, we find that $C \geq \big[\frac{p-1}{p}-\alpha\big]_q^p$.
Moreover, if $C > 0$ is a sharp constant in inequality (\ref{5.7}), then
$$
C \geq \left(\frac{1-q^{1-\alpha+\beta}}{1-q}\right)^p\frac{1-q}{1-q+\frac{1-q^{1+p\beta}}{[\frac{p-1}{p}-\alpha]_q}}.
$$
Again, by letting $\beta\rightarrow -\frac{1}{p}$, we obtain $C \geq \big[\frac{p-1}{p}-\alpha\big]_q^p$.
The proof is complete.
\end{proof}

%%%%%%%%%%%%% Remark 2.5
\begin{remark}  From Theorem \ref{th3.1} with $\alpha=0$ we obtain the $q$-analogue of the classical Hardy inequality
\begin{equation*}
\int_0^\infty \big(\frac{1}{x}\int\limits_0^x f(t) \, d_qt \big)^p d_qx \leq
\frac{1}{[\frac{p-1}{p}]_q^p} \int_0^\infty f^p(t) \, d_qt,\ f \geq 0,
\end{equation*}
if  $p>1$ or $p<0$ and $f>0$. Moreover, the constant $\frac{1}{[\frac{p-1}{p}]_q^p}$ is best possible
and $\frac{1}{[\frac{p-1}{p}]_q^p} < (\frac{p}{p-1})^p$.
\end{remark}

%%%%%%%%%%%%% Remark 2.6
\begin{remark} If $f \geq 0$ is a continuous function on $[0, 1]$, then by passing to the limit as $q\rightarrow1^-$ 
in (\ref{5.6}) and (\ref{5.7}) we get
\begin{equation}\label{5.17}
\int_0^1 f^p(t)(1-t^{\frac{p-1}{p}-\alpha}) \, dt \leq
\big(\frac{p-1}{p}-\alpha\big)^p\int_0^1 x^{p(\alpha-1)}\big(\int_0^x t^{-\alpha} f(t) \, dt \big)^p dx,
\end{equation}
and
\begin{equation*}
\int_0^1f^p(t) \, dt \leq
\big(\frac{p-1}{p}-\alpha\big)^p \int_0^1 x^{p(\alpha -1)}\big(\int_0^x t^{-\alpha} f(t) \, dt \big)^p d\mu(x)
\end{equation*}
where
$d\mu(x) = \left(1+\frac{p}{p-p\alpha-1} \, \delta(1-x)\right) dx$ and $\delta(\cdot)$ is a Dirac delta function.
\end{remark}
The inequality (\ref{5.17}) is one of the cases recently proved in \cite[Theorem 2.4 (b)]{PS12}.

%%%%%%%%%%%%%%%%%%%%%%%%%%%%%%%%%%%%%% Section 3
\begin{section}
{\bf A new sharp inequality for the Riemann-Liouville operator in $q$-analysis}
\end{section}

We need definitions and formulas from the $q$-calculus to be able to define a
$q$-analogue of fractional integration Riemann-Liouville operator of order ${\alpha>0}$.
These facts are taken mainly from the book  \cite{KC02} (see also \cite{Al66} and \cite{SRM09}).

If $x\ge t >0$, then the $q$-analogue of the polynomial $(x-t)^k$ of order $k\in \mathbb{N}$
and the generalized polynomial $(x-t)^\alpha$ of order $\alpha\in\mathbb{R}$ are defined by
the following relations
\begin{equation}\label{2.5}
(x-t)_q^k = x^k\ (\frac{t}{x}; q)_k \,\,\,  {\rm and } \,\, \, (x-t)_q^\alpha = x^\alpha (\frac{t}{x}; q)_\alpha,
\end{equation}
respectively, where a $q$-analogue of the Pochhammer symbol ($q$-shifted factorial) is defined by
$$
(a; q)_0 = 1, ~ (a; q)_k = \prod\limits_{i=0}^{k-1}(1-aq^i) ~ {\rm for} ~ k \in {\mathbb N} \cup \{\infty\} ~ {\rm and} ~
 (a;q)_\alpha=\frac{(a; q)_\infty}{(aq^\alpha; q)_\infty }.
$$
In $q$-analysis the gamma function $\Gamma _q$ has the form
\begin{equation*}
\Gamma _q(x)= \frac{(q; q)_\infty}{(q^x;q)_\infty}(1-q)^{1-x},\ {\rm for} \, x\in\mathbb{R}\setminus \{0, -1, -2, \ldots\},
\end{equation*}
and the beta function $B_q(\cdot,\cdot)$  is defined in the following way:
\begin{equation*}\label{2.7}
B_q(a, b) = \int_0^1 t^{a-1}(qt; q)_{b-1} \, d_qt = (1-q) \sum_{i=0}^\infty q^{ia}(q^{i+1};q)_{b-1}.
\end{equation*}
Moreover, the following relations are valid:
\begin{equation*}\label{2.8}
\Gamma_q(x+1)=[x]_q \, \Gamma_q(x) \,\,\, {\rm and} \,\,\, B_q(a, b)=\frac{\Gamma_q(a)
\Gamma_q(b)}{ \Gamma_q(a+b)}.
\end{equation*}
Finally, the $q$-analogue of the fractional integration Riemann-Liouville
operator of order ${\alpha>0}$ has the form
\begin{equation}\label{2.9}
I_q^\alpha f(x) = \frac{1}{\Gamma_q(\alpha)}\int\limits_0^x(x-q t)^{\alpha-1}_q f(t) \,d_qt.
\end{equation}

Our main result in this section is the following $q$-analogue of the inequality (\ref{1.2}).

%%%%%%%%%%%%%%%%% Thm 3.1
\begin{theorem}\label{th4.1}
If $p>1$ and $\alpha>0$, then the inequality
\begin{equation}\label{4.1}
\int_0^\infty \left[\frac{1}{x^\alpha\Gamma_q(\alpha)}\int\limits_0^x
(x- q t)^{\alpha-1}_q f(t) \, d_qt \right]^p d_qx \leq C \int_0^\infty f^p(t) \, d_qt,\ f \geq 0,
\end{equation}
holds with the best constant
\begin{equation}\label{4.2}
C= \left [\frac{\Gamma_q(1-\frac{1}{p})} {\Gamma_q(\alpha+1-\frac{1}{p})} \right]^p.
\end{equation}
\end{theorem}

%%%%%%%%%%%%%
\begin{proof} Let $f \geq 0$. Based on (\ref{2.1}), (\ref{2.5}) and (\ref{2.9}) we have
\begin{equation}\label{4.3}
I_q^\alpha f(x)= \frac{1}{\Gamma_q(\alpha)}\int\limits_0^x (x - q t)^{\alpha-1}_q f(t) \, d_qt=
\frac{x^\alpha}{\Gamma_q(\alpha)}(1-q)\sum_{i=0}^\infty \left(q^{i+1}; q\right)_{\alpha-1}f(xq^i) \,q^i.
\end{equation}
Then, in view of (\ref{2.2}) and (\ref{4.3}), we find that
\begin{eqnarray*}
\int_0^\infty \left(\frac{I_q^\alpha f(x)}{x^\alpha}\right)^p d_q x
&=&
(1-q)\left(\frac{1-q}{\Gamma_q(\alpha)}\right)^p\sum_{j=-\infty}^\infty\left(\sum_{i=0}^\infty
(q^{i+1};q)_{\alpha-1}f(q^{i+j})q^i\right)^pq^j \\
&=&
(1-q)\left(\frac{1-q}{\Gamma_q(\alpha)}\right)^p\sum_{j=-\infty}^\infty
q^{j(1-p)}\left(\sum_{i=j}^\infty
(q^{i-j+1};q)_{\alpha-1}f(q^i)q^i\right)^p \\
&=&
(1-q)\left(\frac{1-q}{\Gamma_q(\alpha)}\right)^p \, (J_\alpha)^p.
\end{eqnarray*}
By applying the duality principle in $l_p({\mathbb Z})$ and by using the H\"{o}lder-Rogers  inequality we obtain
\begin{eqnarray*}
J_\alpha
&=&
\sup\limits_{\|g\|_{l_{p'}}=1, \ g\ge 0} \sum_{j=-\infty}^\infty g_j q^{-\frac{j}{p'}}
\sum_{i=j}^\infty \left(q^{i-j+1};q\right)_{\alpha-1}f(q^i) \, q^i \\
&=&
\sup\limits_{\|g\|_{l_{p'}}=1, \ g\ge 0} \sum_{j}\sum_{i} g_j q^{\frac{i-j}{p'}}\theta(i-j)
\left(q^{i-j+1};q\right)_{\alpha-1}f(q^i) \, q^\frac{i}{p} \\
&\leq&
\sup\limits_{\|g\|_{l_{p'}}=1, \ g\ge 0}\left(\sum_{j}\sum_{i} g_j^{p'} q^{\frac{i-j}{p'}}\theta(i-j)
\left(q^{i-j+1};q\right)_{\alpha-1} \right)^\frac{1}{p'} \\
&\times&
\left( \sum_{i}\sum_{j} f^p(q^i)q^i q^{\frac{i-j}{p'}}\theta(i-j) \left(q^{i-j+1};q\right)_{\alpha-1} \right)^\frac{1}{p}
=\sup\limits_{\|g\|_{l_{p'}}=1, \ g\ge 0} J_{\alpha,p'} (g) \, J_{\alpha,p} (f),
\end{eqnarray*}
where
$$
J_{\alpha,p'}(g)^{p'} = \sum_{j} \sum_{i} g_j^{p'} q^{\frac{i-j}{p'}} \theta(i-j) \left(q^{i-j+1};q\right)_{\alpha-1} 
$$
and
$$
J_{\alpha,p}(f)^{p} = \sum_{i}  \sum_{j} f^p(q^i) q^i q^{\frac{i-j}{p'}} \theta(i-j) \left(q^{i-j+1};q\right)_{\alpha-1}.
$$
By formulas on beta and gamma functions, we get
\begin{eqnarray*}
\sup\limits_{\|g\|_{l_{p'}}=1, \ g\ge 0}  J_{\alpha,p'}(g)^{p'}
&=&
\sup\limits_{\|g\|_{l_{p'}}=1, \ g\ge 0} \sum_{j} \sum_{i} g_j^{p'} q^{\frac{i-j}{p'}} \theta(i-j) \left(q^{i-j+1};q\right)_{\alpha-1} \\
&=&
\sup\limits_{\|g\|_{l_{p'}}=1, \ g\ge 0} \sum_{j} g_j^{p'} \sum_{i=j}^\infty q^{\frac{i-j}{p'}}\left(q^{i-j+1};q\right)_{\alpha-1} \\
&=&
\sup\limits_{\|g\|_{l_{p'}}=1, \ g\ge 0} \sum_{j} g_j^{p'} \sum_{i=0}^\infty q^{\frac{i}{p'}}\left(q^{i+1};q\right)_{\alpha-1} \\
&=&
\frac{B_q(\frac{1}{p'};\alpha)}{1-q} = \frac{\Gamma_q(1-\frac{1}{p})
\Gamma_q(\alpha)}{\Gamma_q(\alpha+1-\frac{1}{p})}\frac{1}{1-q},
\end{eqnarray*}
\vspace{-3mm}
and
\vspace{-3mm}
\begin{eqnarray*}
J_{\alpha,p}(f)^p
&=&
 \sum_{i}  \sum_{j} f^p(q^i) q^i q^{\frac{i-j}{p'}} \theta(i-j) \left(q^{i-j+1};q\right)_{\alpha-1}\\
 &=&
\sum_{i} f^p(q^i)q^i \sum_{j=-\infty}^i q^{\frac{i-j}{p'}}\left(q^{i-j+1};q\right)_{\alpha-1} \\
&=&
\sum_{i} f^p(q^i)q^i \sum_{j=0}^\infty q^{\frac{j}{p'}}\left(q^{j+1};q\right)_{\alpha-1} \\
&=&
\frac{1}{(1-q)^2} \frac{\Gamma_q(1-\frac{1}{p})\Gamma_q(\alpha)}{\Gamma_q(\alpha+1-\frac{1}{p})}
\int_0^\infty f^p(t) \, d_qt.
\end{eqnarray*}
By combining the above calculations we find that for $ f \geq 0$ we have
\begin{eqnarray*}
\int_0^\infty \left(\frac{I_q^{\alpha} f(x)}{x^\alpha}\right)^p \,d_qx
&=&
\int_0^\infty \left(\frac{1}{x^\alpha\Gamma_q(\alpha)}\int\limits_0^x (x - qt)^{\alpha-1}_q f(t) \, d_qt\right)^pd_qx \\
&=& 
(1-q) \left( \frac{1-q}{\Gamma_q(\alpha)} \right)^p \, (J_{\alpha})^p \\
&\leq& 
(1-q) \left( \frac{1-q}{\Gamma_q(\alpha)} \right)^p \,  \sup\limits_{\|g\|_{l_{p'}}=1, \ g\ge 0} J_{\alpha,p'}(g)^p \, J_{\alpha, p}(f)^p\\
&\leq&
(1-q) \left( \frac{1-q}{\Gamma_q(\alpha)} \right)^p \left[ \frac{\Gamma_q(1-1/p) \Gamma_q(\alpha)}{\Gamma_q(\alpha + 1 - 1/p)} \frac{1}{1-q}\right]^{p-1} \\
&\times& 
\frac{1}{(1-q)^2} \frac{\Gamma_q(1-1/p) \Gamma_q(\alpha)}{\Gamma_q(\alpha + 1 - 1/p)} \int_0^{\infty} f^p(t)\, d_qt\\
&=&
\left[ \frac{\Gamma_q(1-\frac{1}{p})}{\Gamma_q(\alpha+1-\frac{1}{p})} \right]^p  \int_0^\infty f^p(t) \, d_qt,
\end{eqnarray*}
which  means that the inequality (\ref{4.1}) holds with the estimate
$ C \leq \left[ \frac{\Gamma_q(1-\frac{1}{p})}{\Gamma_q(\alpha+1-\frac{1}{p})} \right]^p$
for the best constant $C$.

Now, we give also a lower estimate for the best constant $C$ in (\ref{4.1}). Let $f_\beta(t) = t^\beta \chi_{(0, 1]}(t)$ with 
$ \beta > - \frac{1}{p}$. Then $ \int_0^\infty f^p_\beta(t) \, d_q t = \frac{1-q}{1-q^{1+ p \beta}}$ (cf. proof of Theorem \ref{th3.1} 
in the case $1 < p < \infty$) and
\begin{eqnarray*}
\int_0^\infty \left(\frac{I_q^\alpha f_\beta (x)}{x^\alpha}\right)^p d_q x
&=&
\frac{(1-q)^{p+1}}{\Gamma^p_q(\alpha)} \sum_{j=-\infty}^\infty q^{j(1-p)} \left(\sum_{i=j}^\infty
\left( q^{i-j+1};q\right)_{\alpha-1} f_\beta (q^i) \, q^i \right)^p \\
&\geq&
\frac{(1-q)^{p+1}}{\Gamma^p_q(\alpha)} \sum_{j=0}^\infty q^{j(1-p)} \left(\sum_{i=j}^\infty \left( q^{i-j+1};q\right)_{\alpha-1}
f_\beta (q^i) \, q^i \right)^p \\
&=&
\frac{(1-q)^{p+1}}{\Gamma^p_q(\alpha)} \sum_{j=0}^\infty q^{j(1-p)} \left(\sum_{i=j}^\infty \left( q^{i-j+1};q\right)_{\alpha-1}
q^{i(1+\beta)} \right)^p \\
&=&
\frac{(1-q)^{p+1}}{\Gamma^p_q(\alpha)} \sum_{j=0}^\infty q^{j(1+p\beta)}
\left(\sum_{i=0}^\infty \left( q^{i+1};q\right)_{\alpha-1} q^{i(1+\beta)} \right)^p \\
&=&
\frac{1-q}{\Gamma_q^p(\alpha) \left(1-q^{1+p\beta}\right)}B_q^p(\beta+1,\alpha).
\end{eqnarray*}
If inequality (\ref{4.1}) holds with the best constant $C > 0$, then
$$
C \geq \sup\limits_{\beta>-\frac{1}{p}}
\left(\frac{B_q(\beta+1,\alpha)}{\Gamma_q(\alpha)}\right)^p = \left(\frac{B_q(1-\frac{1}{p},\alpha)}{\Gamma_q(\alpha)}\right)^p=
\left(\frac{\Gamma_q(1-\frac{1}{p})}{\Gamma_q(\alpha+1-\frac{1}{p})}\right)^p,
$$
which shows that constant (\ref{4.2}) is sharp. The proof is complete.
\end{proof}

From Theorem \ref{th4.1} we obtain immediately the validity of the following statement:

%%%%%%%%%%%%%%%%%%% 3.2
\begin{corollary}\label{th4.2}
Let $p>1$ and $\alpha>0$. Then the following inequality is valid
$$
\int_0^1 \left(\frac{I_q^\alpha f (x)}{x^\alpha}\right)^p d_q x <
\left(\frac{\Gamma_q(1-\frac{1}{p})}{\Gamma_q(\alpha+1-\frac{1}{p})}\right)^p
\int_0^1 f^p(t) \, d_qt.
$$
Moreover, the constant $\left(\frac{\Gamma_q(1-\frac{1}{p})}{\Gamma_q(\alpha+1-\frac{1}{p})}\right)^p$
is best possible.
\end{corollary}
%%%%%%%%%%%%%%%%%%%%%%%%%%%%%%%%%%%%%%

The strict inequality we are getting as before in the estimate of $J_{\alpha,p}(f)$. In fact, for finite interval
of integration the sum inside of the expression $J_{\alpha}$ is going from $0$ to $\infty$,
\begin{eqnarray*}
J_{\alpha,p}(f)
&=&
\sum_{i=0}^\infty f^p(q^i)q^i \sum_{j=0}^i q^{\frac{i-j}{p'}}\left(q^{i-j+1};q\right)_{\alpha-1} <
\sum_{i=0}^\infty f^p(q^i)q^i \sum_{j=0}^\infty q^{\frac{j}{p'}}\left(q^{j+1};q\right)_{\alpha-1} \\
&=&
\frac{1}{(1-q)^2} \frac{\Gamma_q(1-\frac{1}{p})\, \Gamma_q(\alpha)}{\Gamma_q(\alpha+1-\frac{1}{p})}
\int_0^\infty f^p(t) \, d_qt.
\end{eqnarray*}

%%%%%%%%%%%%%%%%%%%%%%%%%%%%%%%%%%%%%% Section 4
\begin{section}
{\bf Remarks on classical discrete Hardy inequalities}
\end{section}

The Hardy discrete inequality (\ref{1.4}) follows from the Hardy integral inequality (\ref{1.3})
by putting in (\ref{1.3}) a simple nonincreasing function (cf. \cite[p. 248]{HLP52}, \cite[p. 726]{KMP06}
and \cite[p. 155-156]{KMP07}).

Up to now there is no sharp discrete analogue of the Hardy integral inequality (\ref{1.1}) except $\alpha = 0$ 
and this fact was motivation for many authors to establish the following discrete inequalities
\begin{equation}\label{3.21}
\sum_{n=1}^\infty \left(\frac{1}{n^{1-\alpha}}\sum_{k=1}^n
\left[k^{1-\alpha} -(k-1)^{1-\alpha}\right] a_k \right)^p \leq
\left(\frac{(1-\alpha)p}{p-\alpha p -1}\right)^p \sum_{n=1}^\infty
a_n^p,\ a_n\ge 0,
\end{equation}
and
\begin{equation}\label{3.22}
\sum_{n=1}^\infty\left(\frac{1}{\sum\limits_{k=1}^n k^{-\alpha}}\sum_{k=1}^n k^{-\alpha} \,a_k \right)^p
\le \left(\frac{(1-\alpha)p}{p-\alpha
p-1}\right)^p\sum_{n=1}^\infty a_n^p, \ a_n\ge 0.
\end{equation}
For fixed $p > 1$, thanks to a result of Cass and Kratz \cite[Theorem 2]{CK90}, we know that the inequalities
(\ref{3.21}) and (\ref{3.22}) can only hold for $\alpha < 1-1/p$ and if they hold for some $\alpha < 1-1/p$,
then the constant $[(1-\alpha)p/(p-\alpha p-1)]^p = [p/(p - \frac{1}{1-\alpha})]^p$ is best possible since
for $\alpha < 1$ we have
$\lim\limits_{n \rightarrow \infty} \frac{\sum_{k=1}^n [k^{1-\alpha} - (k-1)^{1-\alpha}]}{n[n^{1-\alpha} - (n-1)^{1-\alpha}]}
= \frac{1}{1-\alpha}$ and $\lim\limits_{n \rightarrow \infty} \frac{\sum_{k=1}^n k^{-\alpha}}{n^{1-\alpha}} =
\frac{1}{1-\alpha}$ (see also \cite[pages 374-375]{Ga11}).

Both inequalities were stated by Bennett in \cite[pages 40-41]{Be96} whenever $p > 1, \alpha < 1-1/p$ and $\alpha \leq 0$. 
No proofs were given in \cite{Be96}. The proof of (\ref{3.21}) for $p > 1, \alpha < 1 - 1/p$ and $\alpha < 0$ (for $\alpha = 0$ 
this is just classical discrete Hardy inequality (\ref{1.4})) was given by Bennett \cite[pp. 401-402, 407]{Be98} and
the proof for $p > 1, \alpha < 1 - 1/p$ by Bennett \cite[Theorem 1, pages 31-32]{Be04}, \cite[Theorem 1, pages 804 and 
828-828]{Be06} and Gao \cite[Corollary 3.1]{Ga05}.

Inequality (\ref{3.22}) was proved independently by Gao \cite[Corollary 3.2]{Ga05} and Bennett \cite[Theorem 7]{Be06} for 
$p > 1, \alpha < 1-1/p$ and if either $\alpha \leq -1$ or $0 \leq \alpha < 1$. Moreover, Gao \cite[Theorem 1.1]{Ga08} has shown 
that inequality holds for $p \geq 2$ and $-1/p \leq \alpha \leq 0$ or $1 < p \leq 4/3$ and $- 1 \leq \alpha \leq -1/p$. 
In \cite[Theorem 6.1]{Ga10} he extended the proof to $p \geq 2$ and $- 1 \leq \alpha \leq 0$. This means that they are still some 
regions with no proof of (\ref{3.22}).

Now, let us comment which discrete Hardy inequalities we are getting from the Hardy $q$-inequalities.
Directly from the proof of Theorems \ref{th3.1} and \ref{th3.2} we obtain the following discrete inequalities of
independent interest: for $0 < q < 1$ and $\alpha < 1-\frac{1}{p}$ we have
\begin{equation*}
\sum_{j=-\infty}^\infty \left(q^{j\left(\alpha+\frac{1}{p} - 1\right)}\sum_{i=j}^\infty q^{i\left(1-\frac{1}{p}-\alpha\right)}
a_i\right)^p \leq \frac{1}{(1-q^{1 - \frac{1}{p}-\alpha})^p} \sum_{i=-\infty}^\infty a_i^p, ~ a_i \geq 0,
\end{equation*}
\begin{equation*}
\sum_{j=0}^\infty \left(q^{j\left(\alpha+\frac{1}{p} - 1\right)}\sum_{i=j}^\infty q^{i\left(1-\frac{1}{p}-\alpha\right)}
a_i\right)^p \leq \frac{1}{(1-q^{1- \frac{1}{p}-\alpha})^p} \sum_{i=0}^\infty a_i^p ~ a_i \geq 0.
\end{equation*}
if either $p>1$ or $p<0$ and $a_i > 0 \, (i \in \mathbb Z$ or $i \in {\mathbb N} \cup \{0\}$, respectively) with the best
constant $(1-q^{1-\frac{1}{p}-\alpha})^{-p}$.

The above two inequalities we can rewrite by putting $\lambda = 1 - \frac{1}{p} - \alpha > 0$ to the following new
sharp discrete inequalities: if  $0 < q < 1$ and either $p > 1$ or $p < 0$ and
$a_n > 0$ \, ($n \in \mathbb Z$ or $n \in {\mathbb N} \cup \{0\}$, respectively), then with the best constant

\begin{equation}\label{3.23}
\sum_{n=-\infty}^\infty \left(\frac{1}{q^{\lambda n}} \sum_{k=n}^\infty q^{\lambda k} \, a_k \right)^p \leq
\frac{1}{(1-q^\lambda)^p} \sum_{n=-\infty}^\infty a_n^p, ~ a_n \geq 0,
\end{equation}
\begin{equation}\label{3.24}
\sum_{n=0}^\infty \left(\frac{1}{q^{\lambda n}} \sum_{k=n}^\infty q^{\lambda k} \, a_k\right)^p \leq
\frac{1}{(1-q^\lambda)^p} \sum_{n = 0}^\infty a_n^p, ~ a_n \geq 0.
\end{equation}
For $0<p<1$ the inequality (\ref{3.23}) holds in the reverse direction. If $p > 1$, then in view of (\ref{3.23})
and (\ref{3.24}) by passing to the dual inequalities with substitution of $p$ by $p'$ we obtain

\begin{equation}\label{3.25}
\sum_{n=-\infty}^\infty \left(q^{\lambda n}\sum_{k=-\infty}^n q^{-\lambda k} \, a_k \right)^p \leq
\frac{1}{(1-q^\lambda)^p} \sum_{n=-\infty}^\infty a_n^p, ~ a_n \geq 0,
\end{equation}
\begin{equation}\label{3.26}
\sum_{n=0}^\infty \left(q^{\lambda n}\sum_{k=0}^n  q^{- \lambda k} \, a_k \right)^p \leq
\frac{1}{(1-q^\lambda)^p} \sum_{n=0}^\infty a_n^p, ~ a_n \geq 0.
\end{equation}
In recent years the following weighted Hardy and weighted Copson inequalities are frequently investigated (see, e.g.
\cite{Be06}, \cite{Ga05}, \cite{Ga10} and the references given there)
%%%%%%%%%%%%%%%%%
\begin{equation}\label{w1}
\sum_{n=0}^\infty \big( \frac{\sum_{k=0}^n \lambda_k a_k}{\sum\limits_{k=0}^n \lambda_k} \big)^p \leq A \sum_{n=0}^\infty
a_n^p, ~~~
 \sum_{n=0}^\infty\big(\frac{\sum_{k=n}^\infty \lambda_k a_k}{\sum\limits_{k=n}^\infty \lambda_k} \big)^p \leq
B \, \sum_{n=0}^\infty a_n^p,
\end{equation}

\begin{equation} \label{w2}
\sum_{n=-\infty}^\infty \big(\frac{\sum_{k=-\infty}^n \lambda_k \, a_k}{\sum\limits_{k=-\infty}^n\lambda_k} \big)^p
\leq C \sum_{n=-\infty}^\infty a_n^p, ~~ \sum_{n=-\infty}^\infty \big(\frac{\sum_{k=n}^\infty \lambda_k a_k}{\sum\limits_{k=n}^\infty\lambda_k}\big)^p \leq D \sum_{n=-\infty}^\infty a_n^p,
\end{equation}
where $\lambda_n > 0,$ $a_n \geq 0, n \in {\mathbb N} \cup \{0\}$ or $n \in \mathbb Z$, respectively. However, in general,
the best constants in the above inequalities have not been found yet.

If $\lambda_k = k^{-\alpha}$ for $k \in \mathbb N$ and $\lambda_0 = a_0 = 0$, then the first inequality in (\ref{w1})
became (\ref{3.22}) and it holds with the best constant $A = \left(\frac{(1-\alpha)p}{p-\alpha p -1}\right)^p$ for parameters,
which have been mentioned at the beginning of this part.

Since $\sum\limits_{k=n}^\infty q^{\lambda k} = \frac{q^{\lambda n}}{1-q^\lambda}$, the estimates (\ref{3.24}) and
(\ref{3.23}) imply that the second inequalities in (\ref{w1}) and (\ref{w2}) (the Copson inequalities) with
$\lambda_k = q^{\lambda k} \, (0 < q < 1, \lambda > 0, k \in {\mathbb N} \cup \{0\}$ or $k \in \mathbb Z$, respectively)
for $p > 1$ or $p < 0$ hold with the best constants $B = 1$ and $D = 1$, respectively.

Also since $\sum\limits_{k=-\infty}^n q^{- \lambda k} = \frac{q^{- \lambda n}}{1-q^\lambda}$, the estimate (\ref{3.25})
implies the first inequality in (\ref{w2}) (the Hardy inequality) with $\lambda_k = q^{-\lambda k} \, (0 < q < 1, \lambda > 0,
k \in \mathbb Z$) for $p > 1$ with the best constants $C = 1$. In the case $0 < p < 1$ the second inequality in (\ref{w2})
holds in the reverse direction.

Inequality (\ref{3.26}) and the obvious estimate $\sum_{k=0}^n q^{-\lambda k} \geq q^{-\lambda n}$ imply that the first
inequality in (\ref{w1}) holds with $\lambda_k = q^{-\lambda k} \, (0 < q < 1, \lambda > 0,
k = 0, 1, 2, \ldots$) for $p > 1$ with the estimate $A \leq (1-q^{\lambda})^{-p}$ for the best constant.

From the proof of Theorem \ref{th3.3} we obtain that if $\lambda > 0, 0 < q < 1, a_n \geq 0 \,(n = 0, 1, 2, \ldots)$
and $0 < p < 1$, then the following discrete inequalities hold with the best constants
$$
\sum_{n=0}^\infty \left(q^{- \lambda n}\sum_{k=n}^\infty q^{\lambda k} a_k\right)^p >
\frac{1}{(1-q^\lambda)^p} \sum_{n=0}^\infty (1-q^{\lambda n}) \, a_n^p,
$$
and
$$
\sum_{n=0}^\infty \left(q^{- \lambda n}\sum_{k=n}^\infty
q^{\lambda k} a_k\right)^p + \frac{1}{1-q^\lambda}\left(\sum_{n=0}^\infty
q^{\lambda n} a_n\right)^p > \frac{1}{(1-q^\lambda)^p} \sum_{n=0}^\infty a_n^p.
$$

%%%%%%%%%%%%%%%%% matrix
The proof of the $q$-inequality for the Riemann-Liouville operator gives estimates for matrix operators.
In fact, from the proof of Theorem \ref{th4.1} we obtain the following inequalities: if $0 < q < 1, \alpha > 0$
and $p > 1$, then
\begin{equation}\label{4.10}
\sum_{n=-\infty}^\infty
\left(q^{-\frac{n}{p'}}\sum_{k=n}^\infty\left(q^{k-n+1}; q\right)_{\alpha-1} q^{\frac{k}{p'}} \, a_k\right)^p
\leq E \sum_{n=-\infty}^\infty a_n^p, ~ a_n \geq 0,
\end{equation}
\begin{equation}\label{4.11}
\sum_{n=0}^\infty \left(q^{-\frac{n}{p'}}\sum_{k=n}^\infty\left(q^{k-n+1};q\right)_{\alpha-1} q^{\frac{k}{p'}} \, a_k\right)^p
\leq E \sum_{n=0}^\infty a_n^p, ~ a_n \geq 0,
\end{equation}
with the best constant  $E = \big(\sum\limits_{n=0}^\infty q^{\frac{n}{p'}} \left(q^{n+1}; q\right)_{\alpha-1}\big)^p.$

\noindent
Since $ \sum\limits_{k=n}^\infty q^{\frac{k}{p'}}\left(q^{k-n+1};q\right)_{\alpha-1}
= q^{\frac{n}{p'}} \sum\limits_{k=0}^\infty q^{\frac{k}{p'}} \left(q^{k+1}; q\right)_{\alpha-1} $, then denoting
$$
\overline{Q}_n = \sum\limits_{k=n}^\infty q^{\frac{k}{p'}}\left(q^{k-n+1}; q\right)_{\alpha-1}
$$
we can rewrite inequalities (\ref{4.10}) and (\ref{4.11}) in the following forms:
\begin{equation}\label{4.12}
\sum_{n=-\infty}^\infty \left(\frac{1}{\overline{Q}_n}\sum_{k=n}^\infty q^{\frac{k}{p'}} \left(q^{k-n+1};q\right)_{\alpha-1}
a_k\right)^p \leq \sum_{n=-\infty}^\infty a_n^p, ~a_n \geq 0.
\end{equation}
\begin{equation}\label{4.13}
\sum_{n=0}^\infty \left(\frac{1}{\overline{Q}_n} \sum_{k=n}^\infty q^{\frac{k}{p'}} \left(q^{k-n+1};q\right)_{\alpha-1}
a_k \right)^p \leq \sum_{n=0}^\infty a_n^p, ~a_n \geq 0.
\end{equation}
Moreover, by passing to the dual inequality in (\ref{4.10}) and substituting $p$ by $p'$ we obtain
\begin{equation}\label{4.14}
\sum_{n=-\infty}^\infty \left(q^{\frac{n}{p}}\sum_{k=-\infty}^n\left(q^{n-k+1}; q\right)_{\alpha-1}
q^{-\frac{k}{p}} \, a_k\right)^p \leq
 \left(\sum_{n=0}^\infty q^\frac{n}{p}\left(q^{n+1}; q\right)_{\alpha-1}\right)^p \sum_{n=0}^\infty a_n^p, ~a_n \geq 0.
\end{equation}
Since $\sum\limits_{k=-\infty}^n \left(q^{n-k+1}; q\right)_{\alpha-1} q^{-\frac{k}{p}} =
q^\frac{n}{p} \sum\limits_{j=0}^\infty \left(q^{k+1}; q\right)_{\alpha-1} =:Q_n,$ the inequality (\ref{4.14}) can be written
in the form
%%%%%%%%%%%%%%%%%% (37)
\begin{equation}\label{4.15}
\sum_{n=-\infty}^\infty \left(\frac{1}{Q_n}\sum_{k=-\infty}^n \left(q^{n-k+1}; q\right)_{\alpha-1}
q^{-\frac{k}{p}} \,a_k\right)^p \leq \sum_{n=0}^\infty a_n^p, ~ a_n \geq 0.
\end{equation}

Inequalities (\ref{4.12}),(\ref{4.13}) and (\ref{4.15}) are examples of sharp matrix inequalities of the forms

\begin{equation*}
\sum_{n=-\infty}^\infty \big( \frac{\sum_{k=n}^\infty \lambda_{n, k} \,a_k}{\sum\limits_{k=n}^\infty\lambda_{n, k}}\big)^p
\leq \sum_{n=-\infty}^\infty a_n^p, \, \, \, \,
\sum_{n=0}^\infty\big(\frac{\sum_{k=n}^\infty \lambda_{n, k} \,a_k}{\sum\limits_{k=n}^\infty\lambda_{n, k}}\big)^p
\leq \sum_{n=0}^\infty a_n^p,
\end{equation*}
\begin{equation*}
{\rm and} ~~ \sum_{n=-\infty}^\infty\big(\frac{\sum_{k=-\infty}^n \lambda_{n, k} \, a_k}
{\sum\limits_{k=-\infty}^n \lambda_{n, k}}\big)^p \leq \sum_{n=-\infty}^\infty a_n^p, ~ a_n \geq 0,
\end{equation*}
where  $p > 1$ and $\lambda_{n, k} > 0$ are of special form. They are generalizations of three inequalities in (\ref{w1})
and (\ref{w2}).

{\bf Acknowledgements}. The second author was supported by the Scientific Committee of Ministry of Education 
and Science of the Republic of Kazakhstan, grant No.1529/GF, on priority area ``Intellectual potential of the country". 
We thank both referees for some valuable suggestions, which have improved the final version of this paper.

%%%%%%%%%%%%%%%%%%%%%%%%%%%%%%%%%%%%%%%%%%

\noindent
{\footnotesize Lech Maligranda and Lars-Erik Persson, Department of Engineering Sciences and Mathematics\\
Lule\r{a} University of Technology, SE-971 87 Lule\r{a}, Sweden\\
{\it E-mail addresses:} {\tt lech.maligranda@ltu.se}, {\tt lars-erik.persson@ltu.se} }\\

\vspace{-3mm}

\noindent
{\footnotesize Ryskul Oinarov, L. N. Gumilyev Eurasian National University, Munaytpasov st. 5\\
010008 Astana, Kazakhstan\\
{\it E-mail address:} {\tt o\_ryskul@mail.ru} }\\

\end{document}